\documentclass[a4paper,12pt]{amsart}
\usepackage{amsmath,mathrsfs,amssymb,xcolor,MnSymbol}
\usepackage[top=3cm,bottom=3cm,outer=3cm,inner=3cm,marginpar=3cm]{geometry}
\usepackage[utf8]{inputenc}

\input{choiyj.sty}

\newtheorem*{question*}{Question}
\newtheorem*{problem*}{Problem}
\newcommand{\iinner}[1]{\left\llangle{#1}\right\rrangle}

\newtheorem{example}[theorem]{Example}

\begin{document}
\title[Bottom of the spectrum of  complete noncompact K\"{a}hler  manifolds]{Bottom of the spectrum of  complete noncompact  K\"{a}hler manifolds}

\author{Ye-Won Luke Cho}
\address{Research Institute of Natural Science, Gyeongsang National University,
	Jinju,  Gyeongnam, 52828, Republic of Korea}
\email{ywlcho@gnu.ac.kr}

\author{Young-Jun Choi}
\address{Department of Mathematics and Institute of Mathematical Science, Pusan National University, 2, Busandaehak-ro 63beon-gil, Geumjeong-gu, Busan, 46241, Republic of Korea}
\email{youngjun.choi@pusan.ac.kr}


\begin{abstract}
We present a survey on the bottom of the spectrum
of the Hodge Laplacian on complete noncompact K\"ahler manifolds, with particular
emphasis on K\"ahler hyperbolic manifolds and bounded symmetric domains. We  also discuss theorems regarding the upper bounds for the bottom of the spectrum under Ricci and bisectional curvature assumptions, along with rigidity results for manifolds attaining the maximal bottom of the spectrum. Throughout the article, we propose several open problems. 
\end{abstract}

\maketitle

\section{Introduction}
Among various important isometric invariants in Riemannian geometry, the \textit{bottom of the spectrum} of the Hodge Laplacian plays a key role in the study of the global geometry of complete noncompact Riemannian manifolds. The notion is closely related to hyperbolic geometry; the existence of a negative upper bound for the Riemannian sectional curvature implies the positivity of the bottom \cite{McKean1970}, and conversely, complete manifolds with positive bottom of the spectrum enjoy certain global properties of hyperbolic manifolds \cite{Brooks81a, Brooks81b}.

In complex geometry, \textit{K\"ahler hyperbolic} manifolds \cite{Gromov1991} form a large class of noncompact complete K\"ahler manifolds with positive bottom. Important examples of such manifolds include  bounded homogeneous domains equipped with complete K\"ahler-Einstein metrics \cite{Kai_Ohsawa2007}. Despite its significance, there have been only a few studies regarding the bottom of the spectrum of bounded symmetric domains; see \cite{Long_Li2019, Long_Zhang_Lin_Shen2023}. Recently, the authors together with K.-H. Lee \cite{Cho_Choi_Lee2026}  obtained a uniform lower bound for the bottom of the spectrum of bounded symmetric domains in terms of some geometric invariants of the domain, using a refined version of Gromov's basic estimate on K\"ahler hyperbolic manifolds. 

On the other hand, there are various studies on the upper bounds for the bottom of the spectrum under negative lower bounds on different types of curvature, starting with the celebrated work of Cheng \cite{Cheng1975}. Several sharp estimates in the K\"ahler setting such as \cite{Li_Wang2005, Munteanu2009} have also been established. Interestingly, manifolds with the maximal bottom of the spectrum are often isometric to a 	warped product $\mathbb{R}\times N$ equipped with a warped product metric \cite{Li_Wang01, Li_Wang02, Li_Wang2005} and those theorems can be comparable to the splitting theorem of Cheeger-Gromoll \cite{Cheeger_Gromoll72}. It even turns out that a complete K\"ahler manifold with the maximal bottom of the spectrum is biholomorphic to a noncompact quotient of the complex hyperbolic space in certain cases as verified in \cite{Munteanu2009, Munteanu10}.

 In this article, we survey several studies regarding the bottom of the spectrum of complete noncompact K\"ahler manifolds and suggest several open problems.

\subsection*{Acknowledgements}
We would like to thank Kang-Hyurk Lee for introducing us to the topic and for many helpful discussions. The first named author was supported by the National Research Foundation of Korea (RS-2026-25498290) and the second named author was supported by the National Research Foundation of Korea (NRF-2023R1A2C1007227), Samsung Science and Technology Foundation (SSTF-BA2201-01).
\section*{Statements and Declarations}
\subsection*{Competing Interests} The authors state that there is no conflict of interests.
\subsection*{Data Availability Statement} This article has not used any associated data.
\section{The bottom of the spectrum}
Let $(X,g)$ be an $n$-dimensional complete Riemannian manifold. Denote by $\Delta$ the Hodge Laplacian which is the negative of the Laplace-Beltrami operator and by $D(\Delta):= \{ u \in L^2(X) : \Delta u \in L^2(X) \}$ the domain of $\Delta$.
\begin{definition}
The \textit{spectrum} $\sigma(\Delta)$ of the Hodge Laplacian $\Delta$ on $X$ is the set
\[
\sigma(\Delta):=\{\lambda \geq 0: \Delta-\lambda I: D(\Delta)\to L^2(X) ~\text{has no bounded inverse}\}.
\]
The nonnegative number
\[
\lambda_0(X,g):=\inf_{\lambda \in \sigma(\Delta)}{\lambda}
\]
is called the \textit{bottom of the spectrum} of $(X,g)$.
\end{definition}

We shall use the notation $\lambda_0(X):=\lambda_0(X,g)$ whenever the choice of $g$ is clear from the context. It is also well-known that the bottom can be realized as the following Rayleigh quotient:
\begin{align}\label{eqn: Rayleigh quotient}
\lambda_0(X)&=\inf \left\{ \frac{\int_X |\nabla f|^2 dV_g}{\int_X f^2 dV_g} : f \in D(\Delta), ~f \neq 0 \right\}.
\end{align}
Here, $\nabla f$ denotes the gradient of $f$ induced by the metric $g$. Since $(X,g)$ is complete, the space $\mathcal {D}(X)$ of smooth functions with compact support is dense in $D(\Delta)$ with respect to the graph norm
\begin{equation*}
	f\mapsto\norm f_g+\norm{df}_g:=\Big(\int_X f^2 dV_g\Big)^{\frac{1}{2}}+\Big(\int_X |\nabla f|^2 dV_g\Big)^{\frac{1}{2}};
\end{equation*}
see, for example, (3.2) Theorem in Ch.VIII  of \cite{Demailly-Book}. Hence we also have
\begin{align}\label{eqn: Rayleigh quotient smooth}
	\lambda_0(X)&=\inf \left\{ \frac{\int_X |\nabla f|^2 dV_g}{\int_X f^2 dV_g} : f \in \mathcal D(X), ~f \neq 0 \right\}.
\end{align}

Several remarks are in order.
\begin{enumerate}
  \setlength\itemsep{0.4em}
\item It is straightforward to check from (\ref{eqn: Rayleigh quotient}) that $\lambda_0$ is an isometry invariant. But $\lambda_0$ is not a local isometry invariant; if $p:X\to Y$ is a covering map, then $\lambda_0(X)\geq \lambda_0(Y)$. See Theorem A in \cite{Ballman_Polymerakis}.
\item If $(X,g)$ has finite volume, then $1_{X}\in D(\Delta)$ so that $\lambda_0(X)=0$ by (\ref{eqn: Rayleigh quotient}). So $\lambda_0(X,g)>0$ only if $(X,g)$ is a complete noncompact manifold with infinite volume. 
\end{enumerate}

One of the interesting features regarding the bottom of the spectrum is its relation with hyperbolic geometry. If $(X,g)$ is a complete manifold with Ricci curvature bounded from below by a positive constant, then it follows from the theorem of Bonnet-Myers that $X$ is compact. Therefore, $\lambda_0(X)=0$ by the previous remark. On the other hand, the bottom of the spectrum of a simply-connected, negatively curved manifold is always positive.

\begin{theorem}[\cite{McKean1970}]\label{thm;McKean}
If $(X,g)$ is a complete, simply-connected $n$-dimensional Riemannian manifold and the Riemannian sectional curvature is bounded from above by $-k<0$, then $\lambda_0(X,g)\geq \frac{(n-1)^2}{4}k$.
\end{theorem} 
The equality holds for the real hyperbolic spaces; see \cite{Li_Wang2005} for more examples of manifolds attaining the equality.
\begin{proof}
Choose a point $p\in X$ and let $\rho(\cdot):=d(\cdot,p)$ be the distance function on $X$ induced by $g$. Then by the Cartan-Hadamard theorem, $\rho$ is smooth on $X\backslash \{p\}$. The Hessian comparison theorem in  \cite{Greene_Wu} implies
\[
\Delta \rho \leq -(n-1)\sqrt{k}\coth(\sqrt{k}\rho)\leq -(n-1)\sqrt{k}
\]
in the sense of distributions. Then for any $f\in \mathcal{D}(M)$, we get
\begin{align*}
(n-1)\sqrt{k}\int_{X}f^2dV_g&\leq -\int_{X}f^2\Delta \rho dV_g=2\int_{X}f\inner{\nabla f,\nabla \rho}dV_g\leq 2\int_{X}|f|\cdot |\nabla f|dV_g\\
&\leq 2\cdot \Big(\int_{X}f^2dV_g\Big)^{\frac{1}{2}} \cdot \Big(\int_{X}|\nabla f|^2dV_g\Big)^{\frac{1}{2}}
\end{align*}
since $|\nabla\,\rho|=1$. So we obtain
\[
\frac{(n-1)^2}{4}k\leq \frac{\int_{X}|\nabla f|^2dV_g}{\int_{X} f^2dV_g}.
\]
The inequality above together with (\ref{eqn: Rayleigh quotient smooth}) completes the proof.
\end{proof}

 See also \cite{Pinski78, Setti1991} for related results. Conversely, manifolds with positive bottom of the spectrum enjoy several global properties of hyperbolic manifolds; see, for example, \cite{Brooks81a, Brooks81b, Buser82}.

\section{K\"ahler hyperbolic manifolds}
Among various notions of hyperbolicity in complex geometry, the following seems to be well-adapted to the study of the bottom of the spectrum of K\"ahler manifolds. 
\begin{definition}\cite{Gromov1991} \label{def: Kahler hyp.}
A complete noncompact K\"ahler manifold $(X,\omega)$ is \textit{K\"ahler hyperbolic} if there exists a global 1-form $\eta$ on $X$ such that $\omega=d\eta$ on $X$ and
\[
\norm{\eta}_{L^{\infty}}:=\sup_{X} |\eta|_{\omega}<\infty.
\]
\end{definition}
Here, $|\eta|_{\omega}$ denotes the pointwise norm of $\eta$ with respect to the metric $\omega$. Originally, Gromov defined a compact K\"ahler manifold to be K\"ahler hyperbolic in \cite{Gromov1991} if its universal covering equipped with the pullback metric satisfies the condition in Definition \ref{def: Kahler hyp.}. Here, we choose the definition given as above, following \cite{Choi_Lee_Seo2023, Choi_Lee_Seo2026}. See also \cite{Bei_Claudon_Diverio_Trapani2024} for a birationally invariant version of K\"ahler hyperbolicity.

The basic example of a K\"ahler hyperbolic manifold is, of course, the Poincar\'e half-plane.
\begin{example}
Let $\mathbb{H}:=\{z=x+iy\in \mathbb{C}: x>0\}$ be the right-half plane equipped with the Poincar\'e metric 
\[
\omega_{\mathbb{H}}:=\frac{1}{K}\frac{dx\wedge dy}{x^2}=d\Big(-\frac{dy}{Kx}\Big)
\]
of constant Gaussian curvature $-K<0$. Then $\eta:=-\frac{dy}{Kx}$ satisfies
\[
|\eta|_{\omega_\mathbb{H}}^2=Kx^2\cdot \Big(-\frac{1}{Kx}\Big) \cdot \Big(-\frac{1}{Kx}\Big)=\frac{1}{K}=\norm{\eta}^2_{L^{\infty}}<+\infty. 
\]
Hence $(\mathbb{H},\omega_{\mathbb{H}})$ is K\"ahler hyperbolic. 
\end{example}

There are various examples of higher-dimensional K\"ahler hyperbolic manifolds that generalize the Poincar\'e half-plane:
\begin{enumerate} \setlength\itemsep{0.4em}
\item Simply-connected K\"ahler manifold with Riemannian sectional curvature bounded from above by a negative constant \cite{Gromov1991, Ballman06,Chen_Yang18}.
\item Bounded homogeneous domain equipped with the complete K\"ahler-Einstein metric \cite{Kai_Ohsawa2007, Choi_Lee_Seo2026}.
\item Hyperconvex domain in $\mathbb{C}^n$ equipped with the K\"ahler metric of the form $\omega=i\partial\bar{\partial}(-\log(-\beta))$, where $\beta$ is a nonpositive smooth plurisubharmonic exhaustion on the domain \cite{Donnelly1994}.
\item Strictly pseudoconvex domain equipped with the complete K\"ahler-Einstein metric \cite{Choi_Lee_Seo2023}.
\end{enumerate}
The Teichm\"uller space $\mathcal{T}_{g,n}$ of Riemann surfaces of genus $g$ with $n$ punctures is also K\"ahler hyperbolic by \cite{McMullen00}.

Denote by $\Delta:=dd^{\ast}+d^{\ast}d:\Omega^k(X)\to \Omega^k(X)$ the Hodge Laplacian acting on the space $\Omega^k(X)$ of differential $k$-forms on $X$. We also denote by $\mathcal{D}^k(X)$ the space of differential $k$-forms with compact support. The following basic estimate of Gromov is fundamental in the study of K\"ahler hyperbolic manifolds.
\begin{theorem}[\cite{Gromov1991}]\label{thm: Gromov vanishing}
If $(X,\omega)$ is a K\"ahler hyperbolic manifold of complex dimension $n$ and $\varphi\in \mathcal{D}^k(X)$ with $k\neq n$, then there exists a constant $c_{k}>0$ depending only on $k$ and $n$ such that
\[
\iinner{\Delta\varphi,\varphi}=\norm{d\varphi}^2+\norm{d^{\ast}\varphi}^2
\ge
\frac{c_k}{\norm{\eta}^2_{L^\infty}}
\norm{\varphi}^2.
\]
\end{theorem}

Since $d^{\ast}\varphi=0$ and $\norm{d\varphi}=\norm{\nabla\varphi}$ for $\varphi\in \mathcal{D}(X)$, the theorem above together with (\ref{eqn: Rayleigh quotient smooth}) implies the positivity of the bottom of the spectrum of K\"ahler hyperbolic manifolds:
\begin{equation}\label{ineq:bottom lower bound Kahler hyp.}
\lambda_0(X,\omega)\geq \frac{c_0}{\norm{\eta}^2_{L^\infty}}>0.
\end{equation}
We also refer the reader to 0.4.A Theorem in \cite{Gromov1991} for some topological consequences of the theorem. To obtain the estimate, Gromov considers the norm of $\psi:=L^{n-k}\varphi\in \mathcal{D}^{2n-k}(X)$ when $k<n$,  where $ L(\varphi):=\omega\wedge \varphi$  denotes the Lefschetz map. In particular, the Leibniz rule
\begin{align*}
	\psi=\varphi\wedge\omega^{n-k}
	&=
	d\paren{\varphi\wedge\eta\wedge\omega^{n-k-1}}
	-
	d\varphi\wedge\eta\wedge\omega^{n-k-1}
	=d\theta-\psi'
\end{align*}
is crucial in the proof. Here, $\theta:=\varphi\wedge\eta\wedge\omega^{n-k-1}$ and $\psi':=d\varphi\wedge\eta\wedge\omega^{n-k-1}$. If $k>n$, apply the Hodge star operator to $\varphi$ and proceed as before.

 To the best of the authors' knowledge,  an explicit formula for the constant $c_k$ first appeared in \cite{Cho_Choi_Lee2026} where we used the aforementioned Gromov's method. The constant $c_0=n^2/4$ obtained in \cite{Cho_Choi_Lee2026} in particular turns out to be a sharp constant; see the remark below Theorem \ref{thm; complex McKean}. Therefore, we conclude from (\ref{ineq:bottom lower bound Kahler hyp.}) the following sharp lower bound for the bottom of the spectrum of K\"ahler hyperbolic manifolds. 
\begin{theorem}[\cite{Gromov1991} \cite{Cho_Choi_Lee2026}]\label{thm: bottom Kahler hyperbolic mflds}
If $(X,\omega)$ is K\"ahler hyperbolic with $\omega=d\eta$, $\norm{\eta}_{L^\infty}<+\infty$, then
\[
\lambda_0(X,\omega)\geq \frac{n^2}{4\norm{\eta}^2_{L^\infty}}.
\]
\end{theorem} 

If $\omega$ admits a global K\"ahler potential, then the theorem reduces to Proposition 2.1 in \cite{Li_Tran2010}. 
\begin{remark}
One can also define the bottom of the spectrum $\lambda_0^k(X)$ of $\Delta$ on $\Omega^k(X)$ as
\[
\lambda_0^k(X)= \inf\set{\frac{\iinner{\Delta\varphi,\varphi}}{\norm{\varphi}^2}:\varphi\in\mathcal{D}^k(X)}.
\]
Then Theorem 1.2 in \cite{Cho_Choi_Lee2026} implies the following lower bound for $\lambda_0^k$ on K\"ahler hyperbolic manifolds:
\begin{equation*}
	\lambda_0^k(X)\ge \frac{c_k}{\norm{\eta}^2_{L^\infty}}>0,
\end{equation*}
where $c_k$ is an explicit constant depending only on $k, n$. But it seems that the lower bound is not sharp if $k\geq 1$. See Remark 4.4 in \cite{Cho_Choi_Lee2026}.
\end{remark}

\section{K\"ahler potentials of minimal constant gradient length on bounded symmetric domains}
 
 The greatest lower bound for the bottom of a K\"ahler hyperbolic manifold $(X,\omega)$ obtainable from Theorem \ref{thm: bottom Kahler hyperbolic mflds} is determined by the number
 \[
\mathsf{M}(X,\omega):=\inf\bigg\{\norm{\eta}^2_{L^{\infty}}:\eta~\text{is a global 1-form on X such that}~ \omega=d\eta\bigg\}.
 \]
 We shall call $\mathsf{M}(X,\omega)$ the $\textit{K\"ahler hyperbolicity modulus}$ of $(X,\omega)$, following \cite{Choi_Lee2026}. A natural candidate $\eta$ that realizes $L_X$ is the one with the constant pointwise length $|\eta|_{\omega}$. If $\varphi$ is a global K\"ahler potential of $(X,\omega)$ so that $\omega=dd^c\varphi$ and $\eta:=d^c\varphi$ has constant norm, then $\varphi$ is a constant gradient length potential. 
 
 Kai-Ohsawa \cite{Kai_Ohsawa2007}  showed that there exists a constant gradient length potential of the complete K\"ahler-Einstein metric (which equals to the Bergman metric in this case) on any bounded homogeneous domain $\Omega$, using the fact that  $\Omega$ is  biholomorphic to a homogeneous Siegel domain $D$ that admits the constant gradient length potential.  The following theorem generalizes the aforementioned theorem for local K\"ahler potentials on bounded symmetric domains.
   
   \begin{theorem}[Theorem 1.4 in \cite{Choi_Lee_Seo2026}, Theorem 6.1 \cite{Cho_Choi_Lee2026}]\label{thm:CLS}
   	Let $\Omega=\Omega_1\times\cdots\times \Omega_s$ be a product of irreducible bounded symmetric domains and $\omega_{\textup{KE}}$ the complete K\"ahler-Einstein metric on $\Omega$ with $-k<0$ being the supremum of the holomorphic sectional curvature of $\omega$.
   	\begin{enumerate}
\setlength\itemsep{0.4em}
   	\item    	If there is a local potential function $\varphi$ of $\omega$ with constant gradient length, then   	$|\partial \varphi|^2_{\omega}=1/k$ .
   	 \item For any global 1-form $\eta$ on $\Omega$ such that $\omega=d\eta$, we have
   	\[
   	\norm{\eta}^2_{L^{\infty}}\geq \frac{1}{k}.
   	\]
   	\end{enumerate}
   \end{theorem} 
    This implies that $\mathsf{M}(\Omega,\omega_{\textup{KE}})=1/k$. An important observation in the proof of Theorem \ref{thm:CLS} is that each integral curve of the local vector field $\mathcal{V}:=\nabla^{1,0}\varphi$ is a totally geodesic holomorphic curve with constant Gaussian curvature. Surprisingly, it suffices to compute the norm of $\partial\varphi$ on such a curve with respect to the metric $\omega_{\textup{KE}}$ restricted to the curve; see \cite{Choi_Lee_Seo2026} for details.
   
   Theorem \ref{thm:CLS} together with Theorem \ref{thm: bottom Kahler hyperbolic mflds} yields the following K\"ahler version of McKean's theorem on bounded symmetric domains.
   \begin{theorem} \cite{Cho_Choi_Lee2026} \label{thm; complex McKean}
   	If $(\Omega,\omega_{\textup{KE}})$ is a bounded symmetric domain equipped with the complete K\"ahler-Einstein metric $\omega_{\textup{KE}}$ with holomorphic sectional curvature  bounded from
   	above by $-k<0$, then
   	\[
   	\lambda_0(\Omega)\geq \frac{n^2}{4}k.
   	\]
   \end{theorem}
   The constant $k$ is a simple function of the geometric invariants (rank and the genus) of the domain; see Section 6 in \cite{Cho_Choi_Lee2026}.  As expected, the lower bound is sharp as it becomes the bottom when the domain is the unit ball or a polydisc. There are several results regarding upper bounds on the bottom of the spectrum in the K\"ahler setting; see Section \ref{Section: upper bound}. But the sharp lower bound for the bottom of the spectrum of bounded symmetric domains we present here seems to be new. 
   
   At this juncture,  it would be natural to present the following question in the light of the proof of Theorem \ref{thm;McKean}.
   \begin{question*}
   	Can one establish a Laplacian comparison theorem on bounded symmetric domains to prove Theorem \ref{thm; complex McKean}?
   	  \end{question*}
        
    We also remark that Ballmann's arguments for the K\"ahler hyperbolicity of simply-connected, negatively curved K\"ahler manifolds together with Gromov's estimate yields a lower bound similar to McKean's bound.
    \begin{theorem}[8.4 Proposition in \cite{Ballman06}]
    	If $(X,\omega)$ is a simply-connected complete K\"ahler manifold with Riemann sectional curvature bounded from above by $-k<0$, then there exists a global 1-form $\eta$ on $X$ such that $\omega=d\eta$ and 
    	\[
    	\norm{\eta}^2_{L^\infty}\leq \frac{1}{k}
    	\] 
    	so that
    	\[
    	 \lambda_0(X)\geq \frac{n^2}{4}k
    	\]
    \end{theorem}
    The idea of the proof is to use the exponential map and the Poincar\'e lemma to construct the 1-form $\eta$. Then the Rauch comparison theorem implies the desired estimate on $\eta$; see Lemma 3.2 in \cite{Chen_Yang18} for the proof. As the lower bound is even smaller than McKean's bound, we present the  following
    \begin{question*}
    Is it possible to formulate a K\"ahler version of the Rauch comparison theorem and sharpen the estimate?
    \end{question*}
    To answer the question, it suffices to control the norm of the $J\gamma'$-component of a Jacobi field on a given geodesic $\gamma$ in terms of the holomorphic (bi)sectional curvature. But this seems to be nontrivial.

\section{Upper bounds on the bottom of the spectrum}\label{Section: upper bound}
A line of research concerning the upper bound of $\lambda_0$ starts with the following celebrated
\begin{theorem}[\cite{Cheng1975}] \label{thm; Cheng}
If $(X,g)$ is a complete Riemannian manifold
of dimension $n$ with the Ricci curvature bounded from below by a negative number $-(n-1)k$, then
\[
\lambda_0(X)\leq \frac{(n-1)^2k}{4}.
\]
\end{theorem}

Note that the bound is sharp as it is attained by the real hyperbolic space. The proof follows from Cheng's estimate for the first Dirichlet eigenvalue on a relatively compact subdomain $\Omega$ of $X$ since  $\lambda_0(X)$ is the decreasing limit of the first eigenvalues as $\Omega$ exhausts the manifold. See also Theorem 6.1 in \cite{Li12} for the proof using the Bochner technique on the norm of the gradient of the first eigenfunction on $\Omega$.

A K\"ahler version of Cheng's theorem was first proved by Li-Wang in \cite{Li_Wang2005}. It says that, if the holomorphic bisectional curvature $BK_{X}$ of a complete K\"ahler manifold $(X^n,\omega)$ satisfies $BK_X\geq -1$, i.e.,
\[
R_{\alpha\bar{\alpha}\beta\bar{\beta}}\geq -1(1+\delta_{\alpha\beta})
\]
for any unitary frame $\{e_{\alpha}\}$, then $\lambda_0(X)\leq n^2$. As $\lambda_0$ determines a lower bound for the growth of the volume of geodesic balls \cite{Li_Wang01}, it suffices to  bound the volume from above by the volume of a geodesic ball on the complex hyperbolic space. Then the authors use the Bochner technique on the norm of the gradient of the distance function to obtain the Laplacian comparison theorem that leads to such bounds. Later, Munteanu \cite{Munteanu2009} improved their result by obtaining the same bound on $\lambda_0$ of complete K\"ahler manifolds with Ricci curvature bounded from below by $-2(n+1)$. As it is not clear whether the aforementioned methods are still available for the case, he uses delicate estimates of the integral

	\[
	\int h\cdot|(\log\,h)_{\alpha\bar{\beta}}|^2\phi^2dV,
	\]
 on certain sets of an end $E$ of $X$ involving the complex hessian. Here, $h$ is a harmonic function and $\phi$ is a bump function. The estimates yield $\lambda_0(E)\leq n^2$ and this settles the claim as the inequality $\lambda_0(X)\leq \lambda_0(E)$ always holds.  We remark that  the aforementioned upper bounds are all sharp in the sense that they are equal to the bottom of the spectrum when the domain is the complex hyperbolic space. 
\subsection{Manifolds with maximal bottom of the spectrum}
An interesting aspect in this line of research is that the isometry types of certain manifolds with the maximal bottom of the spectrum can be determined.
\begin{theorem}[\cite{Li_Wang02}]
	
Let $(X^n,g)$ be a complete $n$-dimensional Riemannian manifold with $n\geq 4$ and more than one end. If the Ricci curvature is bounded from below by $-(n-1)$ and $\lambda_0(X)=\frac{(n-1)^2}{4}$, then $X$ is isometric to $\mathbb{R}\times N$ equipped with the warped product metric
\[
g_X:=dt^2+\exp(2t)g_N
\] 
for some compact manifold $N$ with nonnegative Ricci curvature.
\end{theorem}
  See also \cite{Li_Wang02}.  Since $X$ has one infinite volume end $E$ and a finite volume end $F$ in this case, there exists a positive harmonic function $h>0$ with finite Dirichlet integral such that
  \[
  \sup_{x\in F}h(x)=+\infty,~ ~\inf_{x\in E}h(x)=0.
  \]
  by \cite{Li_Tam92}. Then  gradient estimates on the product of the harmonic function $h$ and a well-chosen bump function together with the Poincar\'{e} inequality
\[
\lambda_0(X) \int_{X}|\phi|^2dV\leq \int_{X}|\nabla\phi|^2dV
\]
imply that $\log\,h$ has constant gradient length. Taking $t:=\log\,h/(n-1)$ as a real coordinate, one can show that the manifold is isometric to the warped product $\mathbb{R}\times N$ equipped with a warped product metric. Here, the manifold $N$ has to be compact as $X$ has more than one end. We refer the reader to the excellent book \cite{Li12} of Peter Li for details.
 
The following theorem of Munteanu provides a characterization of noncompact ball quotients in terms of the Ricci curvature lower bound and the bottom of the spectrum. 
\begin{theorem}[\cite{Munteanu2009}]\label{thm;Mun}
Let $(X^n,\omega)$ be a complete noncompact K\"ahler manifold of complex dimension $n\geq 2$ with more than one end. If the Ricci curvature of $X$ is bounded from below by $-2(n+1)$ and $\lambda_0(X)=n^2$, then it is isometric to $\mathbb{R}\times N$ for some compact manifold $N$ equipped with the warped product metric
\[
g_X=dt^2+e^{-4t}\omega_2^2+e^{-2t}(\omega_3^2+\cdots+\omega^2_{2n}),
\]
where $\{\omega_2,\dots, \omega_{2n}\}$ is an orthonormal coframe for $N$. Furthermore, if $X$ has bounded curvature, then $X$ is a quotient of the complex hyperbolic space.
\end{theorem} 

Note that the manifold must have exactly two ends in this case. See also \cite{Li_Wang09, Munteanu10} for  characterizations of (1) ball quotients under bisectional curvature lower bounds and  (2) simply-connected K\"ahler manifolds with compact quotients, respectively. 

The proof again involves highly nontrivial estimates on the slightly different integral 
	\[
\int h\cdot|(\log\,h)_{\alpha\bar{\beta}}|^2\phi^3dV,
\]
on the manifold, where $h$ is the aforementioned harmonic function constructed in \cite{Li_Tam92}. Then the estimates imply that the function
\begin{equation}\label{eqn; phi}
\varphi:=-\frac{1}{2n}\log\,h
\end{equation}
has constant gradient length so that the manifold splits as the product of $\mathbb{R}$ and a level set of $\varphi$. 

Then it would be natural to ask whether the condition on the number of ends and the boundedness of the sectional curvature in the theorem  are redundant. For the end condition, the example in Theorem 1.3 of \cite{Li10} would provide a counterexample. It says that  any bounded strictly pseudoconvex domain of the form
\[
D(A):=\bigg\{z\in \mathbb{C}^n: |z|^2+\textup{Re}\,\bigg(\sum_{j=1}^nA_jz_j^2\bigg)-1<0\bigg\},~A=(A_1,\dots,A_n)\neq 0,
\]
where $0\leq A_1\leq A_2\leq\dots A_n<1,$ equipped with the complete K\"ahler-Einstein metric on $D(A)$ of Ricci curvature $-2(n+1)$  has the same bottom of the spectrum $\lambda_0(D(A))=n^2$ as the unit ball. Furthermore, the domain $D(A)$ is not biholomorphic to the unit ball so that the automorphism group of $D(A)$ is compact by the Wong-Rosay theorem.  On the other hand,  it is not known whether one can answer the following
\begin{question*}
Can the boundedness condition on the curvature be removed in Theorem \ref{thm;Mun} to obtain the same conclusion?
\end{question*}

Another interesting feature regarding Theorem \ref{thm;Mun} is that it provides a characterization of the complete K\"ahler-Einstein metric $\omega_{\textup{KE}}$ on the given K\"ahler manifold $X$ in terms of the functional $\lambda_0(X,\cdot)$. As the given metric $\omega$ turns out to be K\"ahler-Einstein in the proof when $X$ has bounded curvature, the theorem says that $\omega_{\textup{KE}}$ is the unique maximizer of $\lambda_0(X,\cdot)$ on the set
\[
\mathcal{K}(X):=\{\omega: \omega~\textit{is a complete K\"ahler metric on}~X~\textit{with}~\textup{Ric}(\omega)\geq -\omega\}.
\] 
If the manifold $X$ is compact, then  K\"ahler-Einstein metrics (if any exist) on $X$  can be characterized as maximizers of a certain functional; see, for example, \cite{Berman_Boucksom_Guedj_Zeriahi2013}. But a direct analogue for complete K\"ahler-Einstein metrics on  noncompact K\"ahler manifolds appears not to be available as the global integration involved in such functionals may diverge. In this context, we propose the following
\begin{question*}
For any noncompact K\"ahler manifold equipped with the K\"ahler-Einstein metric $\omega_{\textup{KE}}$ of Ricci curvature $-1$, is $\omega_{\textup{KE}}$ the unique maximizer of the functional $\lambda_0(X,\cdot)$ on $\mathcal{K}(X)$?
\end{question*}
  
Finally, we remark that it would also be interesting to compare the theorem to Yau's characterization \cite{Yau77} of compact quotients of the unit ball with ample canonical line bundle.

\bibliographystyle{alpha} 
\bibliography{referenceschoiyj.bib}

\end{document}